\documentclass{gthack}   
\usepackage[margin=1.4in, left=1.35in, right=1.1in]{geometry}

\usepackage{graphicx}
\usepackage{enumitem}
\usepackage{tikz-cd}

\usepackage{comment}
\usepackage[margin=1cm]{caption}

\title{The trace embedding lemma and spinelessness}

\author[Hayden]{Kyle Hayden} \address{Columbia University, New York, NY 10027} 
\email{hayden@math.columbia.edu}

\author[Piccirillo]{Lisa Piccirillo} 
\address{Brandeis University, Waltham, MA 02453}
\email{piccirli@mit.edu}

\theoremstyle{plain}
\newtheorem{thm}{Theorem}[section]   \newtheorem{lem}[thm]{Lemma}          \newtheorem{prop}[thm]{Proposition}

\newtheorem*{ques*}{Question}

\newtheorem{mainthm}{Theorem}[section]   
\newtheorem{maincor}{Corollary}[mainthm]

\theoremstyle{definition}
\newtheorem{defn}[thm]{Definition}    

\newtheorem{rem}[thm]{Remark}       \newtheorem*{rem*}{Remark}             
\newcommand{\cc}{\mathbb{C}}
\newcommand{\rr}{\mathbb{R}}
\newcommand{\zz}{\mathbb{Z}}
\newcommand{\st}{{\mathrm{st}}}

\newcommand{\lk}{\mathrm{lk}}

\renewcommand{\S}{\textsection}

\newcommand{\cee}{\mathbb{C}}

\newcommand{\cpbar}{\overline{\cee {P}^2}}
\newcommand{\cp}{\cee {P}^2}

\usepackage{caption}

\begin{document}

\begin{abstract}
We demonstrate new applications of the trace embedding lemma to the study of piecewise-linear surfaces and the detection of exotic phenomena in dimension four. We provide infinitely many pairs of homeomorphic 4-manifolds $W$ and $W'$ homotopy equivalent to $S^2$ which have smooth structures distinguished by several formal properties: $W'$ is diffeomorphic to a knot trace but $W$ is not, $W'$ contains $S^2$ as a smooth spine but $W$ does not even contain $S^2$ as a piecewise-linear spine, $W'$ is geometrically simply connected but $W$ is not, and $W'$ does not admit a Stein structure but $W$ does.  In particular, the simple spineless 4-manifolds $W$ provide an alternative to Levine and Lidman's recent solution to Problem 4.25 in Kirby's list. We also show that all smooth 4-manifolds contain topological locally flat surfaces that cannot be approximated by piecewise-linear surfaces. 
\end{abstract}

\maketitle

\vspace{-.6cm}
\section{Introduction}

In 1957, Fox and Milnor observed that a knot $K \subset S^3$ arises as the link of a singularity of a piecewise-linear 2-sphere in $S^4$  with one singular point if and only if $K$ bounds a smooth disk in $B^4$ \cite{fox-milnor:abstract,fox-milnor}; such knots are now called \emph{slice}. Any such 2-sphere has a neighborhood diffeomorphic to the \emph{zero-trace} of $K$, where the \emph{$n$-trace} is the 4-manifold $X_n(K)$ obtained from $B^4$ by attaching an $n$-framed 2-handle along $K$. In this language, Fox and Milnor's observation says that a knot $K \subset S^3$ is slice if and only if $X_0(K)$ embeds smoothly in $S^4$ (cf~\cite{kirby-melvin,miller-picc}). This fact, known as the \emph{trace embedding lemma},  can be combined with work of Freedman \cite{freedman} and Donaldson \cite{donaldson} to give an elegant proof that $\rr^4$ supports exotic smooth structures (cf~\cite[p.522]{GompfStipsicz4}), and it also gives rise to a powerful sliceness obstruction \cite{picc:conway}. 

In this paper, we give new applications which demonstrate how natural extensions of the trace embedding lemma can be parlayed into a tool for detecting exotic smooth structures on small, compact 4-manifolds while also constraining their handle structures and piecewise-linear topology.

\begin{mainthm}\label{thm:main}
For all $n \in \zz$, there exist infinitely many pairs of homeomorphic smooth 4-manifolds $W$ and $W'$ such that $W'$ is the $n$-trace of a slice knot $K \subset S^3$, yet $W$ is not diffeomorphic to the trace of any knot in $S^3$.
\end{mainthm}

We apply Theorem~\ref{thm:main} in three directions: (1) the study of piecewise-linear surfaces in 4-manifolds, (2) the handle decompositions of simply connected 4-manifolds, and (3) the Stein fillings of 3-manifolds obtained by surgery on slice knots in $S^3$.

First, recall that a closed surface $\Sigma$ is called a \emph{spine} of a smooth 4-manifold $X$ if there exists an embedding $\Sigma \hookrightarrow X$ which is a homotopy equivalence. Since each 4-manifold $W'$ from Theorem~\ref{thm:main} is the trace of a slice knot, $W'$ contains $S^2$ as a smoothly embedded spine; this $S^2$ is the union of the slice disk and the core of the 2-handle. In contrast, we use the trace embedding lemma and an adjunction inequality to prove that $W$ cannot even contain $S^2$  as a piecewise-linear spine; we discuss this strategy further in Theorem~\ref{thm:spineless} below.

\begin{maincor}\label{cor:spineless}
Each $W$ is homotopy equivalent to $S^2$ and  contains $S^2$ as a topological locally flat spine but not as a piecewise-linear spine.
\end{maincor}

This provides an alternative solution to Problem 4.25 in Kirby's list \cite{kirby}, which was recently resolved by Levine and Lidman \cite{levine-lidman:spineless}. Levine and Lidman produced examples of smooth, compact 4-manifolds 
that are homotopy equivalent to $S^2$ but do not contain $S^2$ as a PL spine, as detected using the $d$-invariants in Heegaard Floer homology. (Subsequently Kim and Ruberman used surgery theory to show that infinitely many of the examples from \cite{levine-lidman:spineless} contain $S^2$ as a topological spine with cone points.)  In fact, the argument used by Levine and Lidman shows that no smooth 4-dimensional homotopy 2-sphere with the same boundary and intersection form as the examples from \cite{levine-lidman:spineless}  can contain $S^2$ as a PL spine.  In contrast, our results show that the existence of spines depends on the smooth structure; each of our spineless 4-manifolds is homeomorphic to a 4-manifold that admits a smooth spine. 

Next, recall that a simply connected 4-manifold is said to be \emph{geometrically simply connected} if it admits a handle decomposition without 1-handles \cite[Problem 4.18]{kirby}. It remains an important open problem to determine whether every closed, simply connected 4-manifold is geometrically simply connected; see~\cite[Problems~4.88-4.89]{kirby} regarding certain homotopy 4-balls and the 4-dimensional Poincar\'e conjecture. Among manifolds with boundary, work of Gordon \cite{gordon:ribbon} implies that a contractible 4-manifold with non-simply-connected boundary is never geometrically simply connected; see also \cite[p.8]{kirby:4-manifolds} and \cite{dlvvw19}. Levine and Lidman's results show that there are 3-manifolds $Y$  such that every smooth 4-dimensional homotopy 2-sphere with boundary $Y$ and positive intersection form  requires 1-handles \cite[Remark~1.2]{levine-lidman:spineless}. We give the first examples demonstrating that, as expected, geometric simple-connectivity can depend on the smooth structure:

\begin{maincor}\label{cor:gsc}
There exist infinitely many pairs of compact, homeomorphic 4-manifolds $W$ and $W'$ such that $W'$ is geometrically simply connected but $W$ is not.
\end{maincor}

As a third application of Theorem~\ref{thm:main}, we obtain new types of  Stein fillings of 3-manifolds given by surgery on knots in $S^3$. In \cite{conway-etnyre-tosun}, Conway, Etnyre, and Tosun classified the (strong) symplectic fillings of any contact 3-manifold obtained by contact $+1$-surgery on a Legendrian knot $K$ in the standard contact $S^3$: such fillings are precisely the exteriors of Lagrangian disks bounded by $K$ in the standard symplectic $B^4$, and thus have the homology type of $S^1\times B^3$. In particular, for such a filling to exist, the  smooth knot type of $K$ must be slice and the underlying surgery 3-manifold must be $S^3_0(K)$.  Note that this 3-manifold naturally bounds another smooth 4-manifold: the zero-trace $X_0(K)$. However, when $K$ is smoothly slice, $X_0(K)$ can \emph{never} be a symplectic filling of \emph{any} contact structure on its boundary   $S^3_0(K)$; this follows from a standard argument using symplectic caps \cite{eliashberg:caps,etnyre:caps,km:p} and an  adjunction inequality \cite{km:thom, mst:product, fintushel1stern:immersed}. Nevertheless, we complement the classification in \cite{conway-etnyre-tosun} by showing that there exist Lagrangian slice knots $K$ such that $S^3_0(K)$ possesses a second fillable contact structure, one filled by a Stein domain \emph{homeomorphic} to $X_0(K)$.  To the authors' knowledge, these provide the first examples in which zero-surgery on a slice knot $K \subset S^3$ admits a symplectic filling with the homology type of $S^2 \times D^2$.

\begin{maincor}\label{cor:lag}
There exist infinitely many slice knots $K \subset S^3$ such that $S^3_0(K)$ bounds at least two homologically distinct Stein domains in $\cc^2$: (1) an exotic copy of $X_0(K)$ that is not diffeomorphic to any knot trace, and (2) the exterior of a Lagrangian disk in $B^4$.
\end{maincor}

Working now in greater generality we extend the construction underlying Theorem~\ref{thm:main} and invoke the trace embedding lemma for higher genus surfaces (Lemma \ref{lem:traceglory}) to obtain the following generalization of Corollary~\ref{cor:spineless}:

\begin{mainthm}\label{thm:spineless}
For any closed, orientable surface $\Sigma$, there exist infinitely many compact 4-manifolds $W$ such that $W$ is homotopy equivalent to $\Sigma$ and contains $\Sigma$ as a topological locally flat spine, yet $W$ does not contain $\Sigma$ as a piecewise-linear spine. 
\end{mainthm}

Obstructing piecewise-linearly embedded surfaces in a 4-manifold $X$ is made difficult by the fact that a singular  point can be the cone on \emph{any} knot in $S^3$. In order to obstruct piecewise-linear embeddings, the standard trick is to show that there is some \emph{set} of knot invariants which have to take on a special set of values for any knot arising as the link of the singularity, and then show that there is no knot in $S^3$ realizing the entire set of invariants. (For a selection of applications of this strategy to various problems, see \cite{FNOP:almostconcordance} on almost concordance, \cite{HLL:homologyconcordance} on homology concordance, \cite{BL:rationalcuspidal} on rational cuspidal curves, and \cite{levine-lidman:spineless} on spinelessness.)  This clever strategy yields powerful obstructions but is often computationally demanding.

We instead use the trace embedding lemma to show that, for 4-manifolds $X$ which embed smoothly in $S^4$ (or more generally in $\#_n\cc P^2$), the existence of a PL surface in a fixed homology class in $X$ implies the existence of a \emph{smooth} surface of twice the genus in the same class in $X$ (or more generally, in a related class in $X\#_n \overline{\cc P^2})$. Therefore, to prove Theorem \ref{thm:spineless}, it suffices to construct a 4-manifold that embeds smoothly in $S^4$ and possesses a topological locally flat spine of genus $g$,  and then show there is no \emph{smooth} surface of genus $2g$ generating its second homology; we obstruct such smoothly embedded surfaces using an adjunction inequality.  We give a tidy first example of our strategy in  \S\ref{sec:example}.

A key question concerning topological spines remains open:

\begin{ques*}
Is there a compact topological 4-manifold that is  homotopy equivalent to $S^2$ yet does not contain a topologically embedded knot trace inducing the homotopy equivalence?
\end{ques*}

A topological embedding of a knot trace is equivalent to a topologically \emph{tame} embedding of a 2-sphere, i.e.~one that is locally flat away from cone points. We remark that the techniques of this paper appear ill-suited to  this question; by Liem and Venema~\cite{liemvenema}, if $W$ is a simply connected, compact, smooth 4-manifold which embeds smoothly in $S^4$, then every second homology class of $W$ is represented by a locally flat 2-sphere.

We use the examples from Theorem \ref{thm:spineless} to address the classical problem of determining when a topological embedding of a PL $m$-manifold $M$ into a PL $n$-manifold $N$ can be approximated arbitrarily closely (in the compact-open topology) by PL embeddings.  In codimension $n-m\geq 3$, Miller proved that such an approximation is always possible \cite{miller:approx}.  In codimension two,  Giffen \cite{giffen}, Eaton et al~\cite{eaton-et-al}, and Matsumoto \cite{matsumoto:wild} constructed counterexamples for even $n \geq 4$ using non-simply-connected manifolds $M$ and topologically wild embeddings $M \hookrightarrow N$, and Venema established positive results when $M$ is a surface with nonempty boundary \cite{venema}.  In codimension two, the problem long remained open for simply connected $M$ with $n \geq 4$ (and for \emph{any} $M$ when $n$ is odd or when the topological embedding $M \hookrightarrow N$ is topologically tame). We show that, in dimension four, counterexamples are abundant:

\begin{mainthm}\label{thm:approx}
In any smooth 4-manifold, every smoothly embedded, closed, orientable surface $\Sigma$ has a tubular neighborhood containing a topological locally flat embedded surface $\Sigma'$ homotopic to $\Sigma$ such that the topological embedding of $\Sigma'$ cannot be approximated by piecewise-linear embeddings.
\end{mainthm}

The topological spines constructed in \cite{kim-ruberman} also provide examples of topologically tame (but not locally flat) embeddings of $S^2$ in certain 4-manifolds that, by \cite{levine-lidman:spineless}, cannot be approximated by PL embeddings.

Traditionally, many questions about spines of 4-manifolds are concerned with embeddings of arbitrary 2-complexes. It is natural to ask if the smooth subsurface techniques used here can obstruct more general 2-complexes from arising as spines of certain 4-manifolds, even though such 2-complexes need not even contain a closed surface as a subcomplex; this question was posed to the authors directly by Viro \cite{viro:question}. We show that the answer is \emph{yes}. For example:

\begin{mainthm}\label{thm:2-complex}
For any finite 2-complex $C$ with $H_2(C) \cong \zz$,  there exists a compact 4-manifold $X \subset S^4$ that is homotopy equivalent to $C$ yet does not contain  $C$ as a PL  spine.
\end{mainthm}

Our argument applies more broadly to any 2-complex $C$ possessing a nonzero class $\alpha \in H_2(C)$ that has a finite orbit under the action on $H_2(C)$ induced by the group of self-homotopy equivalences of $C$. We expect similar techniques to hold in even greater generality, though we do not pursue that here.

Finally, we recall that the trace embedding lemma can be used to obstruct the existence of a slice disk for a knot $K$; if a non-slice knot $K'$ has $X_0(K)\cong X_0(K')$, then the trace embedding lemma implies that $K$ is not slice. It is natural to ask whether the generalizations of the trace embedding lemma can be used similarly to obstruct the existence of higher genus slice surfaces. For the sake of completeness in the literature, we write down a proof that the immediate answer is \emph{no}: If $g>0$, then any two knots $K,K' \subset S^3$ with diffeomorphic \emph{genus g traces} (as defined in Definition \ref{def:hgh}) must be isotopic; see Proposition~\ref{prop:fail}.

\smallskip

\emph{Organization.} We begin in \S\ref{sec:example} by proving a special case of Theorem~\ref{thm:main} as well as Corollaries \ref{cor:spineless}-\ref{cor:lag}. This short section is intended to provide a simple example illustrating the core ideas and constructions to the casual reader. The formal reader can skip ahead to the later sections where we work in more generality. In \S\ref{sec:definitons}, we define higher genus knot traces, and give the necessary background on trace embeddings and Stein structures. In \S\ref{sec:fake-traces}, we prove Theorem~\ref{thm:detail} (which supercedes Theorems \ref{thm:main} and \ref{thm:spineless}) and Theorem~\ref{thm:approx}. Finally, in  \S\ref{sec:2comlex} we shift our attention to general 2-complexes and prove Theorem \ref{thm:2-complex}.

\smallskip

\emph{Conventions.} All manifolds are assumed to be smooth unless stated otherwise. A map between smooth manifolds $M$ and $N$ is called  \emph{piecewise linear} if there exist triangulations of $M$ and $N$ (compatible with their smooth structures) with respect to which the map is piecewise linear. We say that a map $f: M \to N$ is a \emph{topological embedding} if it is a homeomorphism onto its image. Such a map is called  \emph{tame} if it is topologically locally  equivalent to a  piecewise-linear map (i.e.~up to local homeomorphism), and otherwise it is called \emph{wild}. The space of all continuous maps between $M$ and $N$ is denoted $C^0(M,N)$ and is always equipped with the compact-open topology. We say that a map $f: M \to N$ can be \emph{approximated} by maps in a subset $A \subset C^0(M,N)$ if $f$ lies in the closure of $A$.

\smallskip

\emph{Acknowledgements.} We thank Adam Levine and Tye Lidman for conversations that helped inspire this paper, and Maggie Miller for her close reading and helpful comments on a draft. K.H.~was supported in part by NSF grant DMS-1803584, and L.P. was supported in part by NSF grant DMS-1902735 and in part by funding from the Simons Foundation and the Centre de Recherches Math\'ematiques, through the Simons-CRM scholar-in-residence program.

\section{An illustrative example and proofs of corollaries}\label{sec:example}

To illustrate the core ideas of our construction, highlight the role of the classical trace embedding lemma, and provide a simple counterexample to Problem 4.25 of \cite{kirby}, we begin with a single example: Let $W$  be the 4-manifold from Figure~\ref{fig:Z1stein}. We will show that: \begin{enumerate}
\item $W$ embeds smoothly in $S^4$,
\item $W$ admits a Stein structure, and
\item $W$ is homeomorphic to $W'=X_0(K)$, where $K$ is a slice knot in $S^3$.\end{enumerate}

\begin{figure}\center
\def\svgwidth{.55\linewidth}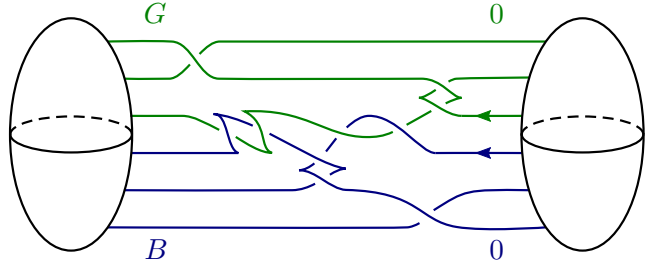
\caption{A Stein handlebody diagram for $W$.}
\label{fig:Z1stein}
\end{figure}

Assuming these for a moment, we demonstrate how the trace embedding lemma shows $W$ is not diffeomorphic to any knot trace. For the sake of contradiction, let us suppose that there is some knot $J \subset S^3$ for which there exists a smooth embedding $X_n(J) \hookrightarrow W$ inducing an isomorphism on second homology. (Note that a diffeomorphism between $X_n(J)$ and $W$ is an example of such an embedding.) Since the intersection form of $W$ is $\langle 0\rangle$, we must have $n=0$. Since $W$ embeds smoothly in $S^4$, the knot trace $X_0(J)$ embeds smoothly in $S^4$, hence the trace embedding lemma implies that $J$ is slice. Thus there is a smooth 2-sphere embedded in $X_0(J)$ generating second homology, which in turn implies that there is a smooth 2-sphere embedded in $W$ generating second homology. However, this violates  the adjunction inequality for Stein 4-manifolds \cite{lisca-matic} as follows: if $\Sigma$ is a smoothly embedded surface in $W$ representing a nontrivial class in $H_2(W)$, then 
\begin{equation}\label{eq:adjunction}
[\Sigma] \cdot [\Sigma] + \left| \langle c_1(W),[\Sigma] \rangle \right| \leq 2g(\Sigma)-2.
\end{equation}
In particular, the smoothly embedded 2-sphere representing a generator of $H_2(W)$  violates \eqref{eq:adjunction}. We conclude that $W$ cannot contain an embedded knot trace generating its second homology, and hence is not diffeomorphic to a knot trace.

\textbf{Proof of (1)}: Consider the diagram of $W$ in dotted circle notation on the left side of Figure~\ref{fig:ExPair}, and let $\mu$ denote a meridian of the dotted circle. We may attach a 2-handle to $W$ along the knot corresponding to the zero-framed curve $\mu$, which can then be canceled with the 1-handle. After erasing the dotted circle and simplifying, we obtain a 0-framed unlink with two components. Attaching two 3-handles and a 4-handle yields $S^4$, exhibiting the desired embedding of $W$ in $S^4$.\hfill $\square$

\textbf{Proof of (2)}: The attaching curves for the 2-handles in Figure~\ref{fig:Z1stein} are drawn as Legendrian knots in the contact boundary of the Stein domain $S^1 \times B^3$. Following~\cite{gompf:stein} (as discussed in Remark~\ref{rem:framings}), we compute the Thurston-Bennequin numbers of the attaching curves to be $tb(G)=1$ and $tb(B)=1$.  Since the 2-handles attached along $G$ and $B$ have framing $tb-1$, it follows that $W$ admits a Stein structure. \hfill $\square$

\begin{figure}\center
\def\svgwidth{.8\linewidth}
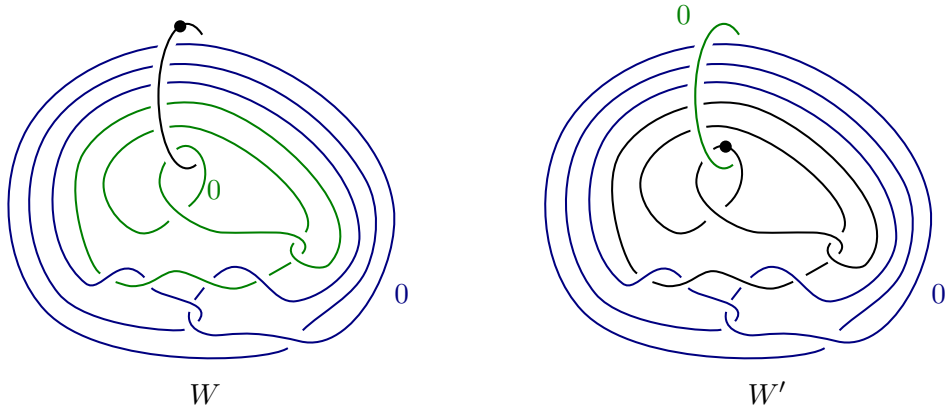
\caption{Diagrams for $W$ and $W'$ related by a dot-zero exchange.}
\label{fig:ExPair}
\end{figure}

\textbf{Proof of (3)}: Let $M$ denote the well-known Mazur cork, defined in Figure \ref{fig:Mazur}, and observe that $W$ contains an embedded copy of $M$ arising from the 0-handle, 1-handle, and the green 2-handle $G$ in Figure~\ref{fig:Z1stein}. That is, $W$ is obtained from $M$ by attaching a 2-handle along a framed knot in $\partial M$ corresponding to the 0-framed curve $B$ in Figure~\ref{fig:Z1stein}. The boundary of $M$ is known to admit an involution $\tau: \partial M \to \partial M$ that extends to a homeomorphism of $M$ (but not a diffeomorphism); see \cite{akbulut:cork,freedman}. The 4-manifold $W'$ obtained by removing $M$ from $W$ and regluing it by $\tau$ is thus homeomorphic to $W$. Diagrammatically, this operation corresponds to a \emph{dot-zero exchange}; see \cite{akbulut:book}. In particular, we first redraw Figure~\ref{fig:Z1stein} in dotted circle notation as on the left side of Figure~\ref{fig:ExPair}. We then obtain a diagram for $W'$ by reversing the roles of the dotted circle and the 0-framed 2-handle attached along $G$ as on the right side of Figure~\ref{fig:ExPair}.

Next we verify that $W'$ is a zero-trace. After isotopy, we obtain the diagram on the left side of Figure~\ref{fig:Zsmooth}.  After performing the indicated handleslide of the green curve $G'$ over $B'$, we obtain a new curve $G''$ as on the right side of Figure~\ref{fig:Zsmooth}. Observe that $G''$ may be isotoped to run over the 1-handle exactly once geometrically. This isotopy produces a pair of intersections between the curve $B'$ and the belt sphere of the 1-handle; these can be eliminated by sliding $B'$ twice over $G''$.  We may then cancel the 1-handle with the 2-handle attached along $G''$, leaving a handle decomposition of $W'$ with exactly one 0-handle and one 0-framed 2-handle. It follows that $W'$ is the zero-trace $X_0(K)$ of some knot $K$ in $S^3$. 

\begin{figure}[b]\center
\def\svgwidth{.85\linewidth}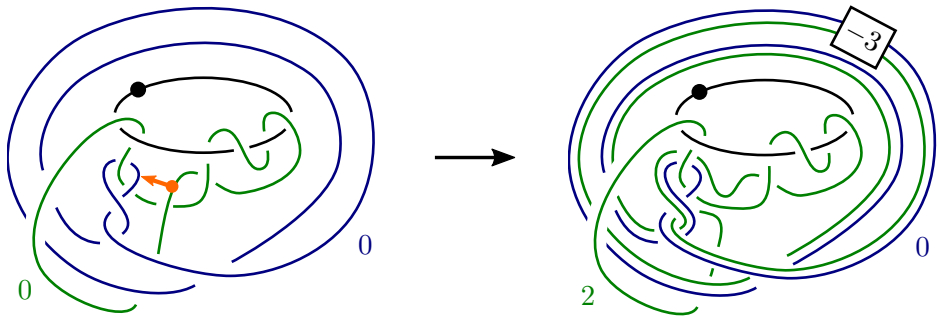
\caption{Simplifying the handle diagram for $W'$.}
\label{fig:Zsmooth}
\end{figure}

It remains to show that $K$ is slice. This can be seen by performing the handle calculus discussed in the previous paragraph, after which it is evident that the knot $K$ admits a ribbon disk with two local minima; we pursue this strategy in the proof of Theorem \ref{thm:detail}. For variety, here we will prove that $K$ is slice by embedding $W'=X_0(K)$ in $S^4$ and applying the trace embedding lemma.

Start with the handle diagram of $W'$ in Figure~\ref{fig:ExPair} and add a 0-framed 2-handle to a meridian $\mu'$ of the green 2-handle $G'$. The blue 2-handle $B'$ can then be slid over the new 2-handle until it is unlinked from $G'$, and we may further isotope it away from the dotted circle to realize $B'$ as a split 0-framed unknot.  We add a 3-handle to cancel this 2-handle. Next, since $G'$ is \emph{homotopic} to a meridian of the dotted circle, we may slide $G'$ over its own 0-framed meridian $\mu'$ to change crossings until $G'$ is \emph{isotopic} to a meridian of the dotted circle.  We may then cancel the 1-handle and the green 2-handle. All that remains is the 0-framed unknot $\mu'$; this may be canceled by adding a final 3-handle. This exhibits an embedding of $W'=X_0(K)$ in $B^4 \subset S^4$, so the trace embedding lemma implies that $K$ is slice. \hfill $\square$

We now turn our attention to the  corollaries of Theorem~\ref{thm:main}, whose proofs require the stronger statement that no embedding of a knot trace into $W$ can induce an isomorphism  on second homology; see Theorem \ref{thm:detail}. 

\begin{proof}[Proof of Corollary \ref{cor:spineless}]
Let $W$ be as in Theorem~\ref{thm:main}. Then $W$ is homeomorphic to the $n$-trace $W'=X_n(K)$ of a slice knot $K$ in $S^3$. Observe that $X_n(K)$ contains a 2-sphere as a \emph{smooth} spine, obtained as the union of the core of the 2-handle and the slice disk for $K$. Therefore $W$ is also homotopy equivalent to $S^2$ and contains a topological locally flat spine given by the image of the smooth spine of $W'=X_n(K)$  under the homeomorphism.

For the sake of contradiction, suppose that $W$ contains a 2-sphere $\Sigma$ as a PL spine.  We claim that $\Sigma \subset W$ has a neighborhood diffeomorphic to the $n$-trace of a knot in $S^3$. To see this, first note that $\Sigma$ may be assumed to be smooth away from  finitely many singular points $p_i$ near which $\Sigma$ is the cone on a knot $K_i \subset S^3$. If we take a path in $\Sigma$ joining the points $p_i$, then a  small tubular neighborhood $V \cong B^4$ of the path meets $\Sigma$ along a singular disk $\Delta$  bounded by the knot $K'=\#_i K_i$ in $\partial V \cong S^3$. Then $\Sigma \smallsetminus \mathring{\Delta}$ is a smooth disk meeting $V$ along $K' \subset\partial V$, and the union of $V$ and a tubular neighborhood of $\Sigma \smallsetminus \mathring{\Delta}$ is a trace of $K'$ with intersection form $\langle n\rangle$ for $n=[\Sigma]\cdot [\Sigma]$, hence is diffeomorphic to $X_n(K')$. Now, since $[\Sigma]$ generates $H_2(W)$, the inclusion $X_n(K') \hookrightarrow W$ induces an isomorphism on second homology. This contradicts Theorem~\ref{thm:main} (as reformulated in Theorem~\ref{thm:detail}).
\end{proof}

\begin{proof}[Proof of Corollary \ref{cor:gsc}] It is well-known that every second homology class of a 4-manifold built without 1-handles can be represented by a piecewise-linear 2-sphere. Indeed, by cellular approximation and transversality, any second homology class is represented by a smoothly embedded surface formed as a union of $n$ disks in the 2-handles (parallel to their cores) and a properly embedded surface in the 0-handle. We may instead join those $n$ disks by $(n-1)$ bands in the boundary of the 0-handle to form a single disk, then cone off its boundary in the 0-handle. The result is a piecewise-linear 2-sphere in the same homology class.

In particular, any geometrically simply connected 4-manifold that is homotopy equivalent to $S^2$ must contain a PL spine. The claim now follows from Theorem~\ref{thm:main} and Corollary~\ref{cor:spineless}.
\end{proof}

\begin{proof}[Proof of Corollary~\ref{cor:lag}] Let $W$  be the 4-manifold from Figure \ref{fig:Z1stein}. By claims (1) and (2) above,  $W$ embeds smoothly in $S^4$  (hence in $B^4$)  and admits a Stein structure.  By work of Gompf \cite{gompf:image}, $W$ can be realized as a Stein subdomain of $\cc^2$ if the almost-complex structure on $W$ induced by the embedding $\iota:W\hookrightarrow B^4 \subset \cc^2$ is homotopic to the almost-complex structure underlying the Stein structure on $W$. Further, again as in \cite{gompf:image}, since $H^2(W,\mathbb{Z})$ has no 2-torsion, this homotopy class is preserved if and only if the orientation and Chern class $c_1(W)$ are preserved. It is straightforward to see that embedding preserves the orientation of $W$, so it remains to compare the Chern classes.  We first consider the almost-complex structure induced by the Stein structure. The evaluation of the Chern class $c_1(W)$ on a generator $\alpha \in H_2(W)$ corresponds to the difference in the rotation numbers of the oriented attaching curves $B$ and $G$ in Figure~\ref{fig:Z1stein}; see \cite{gompf:stein}. These rotation numbers are both 1 (as per Remark~\ref{rem:framings}), so $\langle c_1(W),\alpha \rangle =0$ and thus $c_1(W)=0$. Next observe that the almost-complex structure induced by the embedding $\iota: W\hookrightarrow \cc^2$ has Chern class given by $\iota^*(c_1(\cc^2))$, which vanishes because $c_1(\cc^2)=0$.  
Thus the Chern class induced by the embedding agrees with the Chern class associated to the Stein structure on $W$.  
Applying \cite[Theorem~2.1]{gompf:image}, we conclude that the image of $W$ in $\cc^2$ is isotopic to a Stein subdomain of $\cc^2$.

\begin{figure}
\center
\def\svgwidth{\linewidth}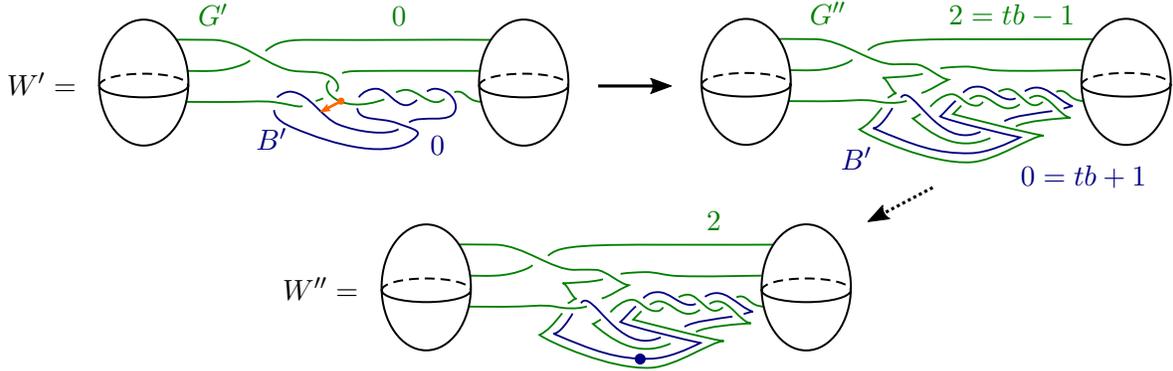 
\caption{Modifying a handle diagram for $W'$   obtain $W''$.}
\label{fig:lag}
\end{figure}

Next we construct a second Stein filling of $\partial W$, obtained as the exterior of a properly embedded Lagrangian disk in $B^4 \subset \cc^2$.  Consider the handle diagram of the 4-manifold $W'$ shown on the left side of Figure~\ref{fig:lag}, which is obtained by simplifying and redrawing the diagram of $W'$ from Figure~\ref{fig:ExPair} in standard 1-handle notation. After performing the indicated handleslide, we obtain the diagram on the right. The attaching curves $G''$ and $B'$ can be viewed as Legendrian knots with $tb(G'')=3$ and $tb(B')=-1$ in the contact boundary of the Stein domain $S^1 \times B^3$. Note that, since $B'$ is a standard Legendrian unknot in the boundary of $S^1 \times B^3$, it bounds a standard Lagrangian disk $\Sigma$ in $S^1 \times B^3$ with Stein exterior. 

Attaching the 2-framed 2-handle to $S^1 \times B^3$ along $G''$ yields a Stein domain diffeomorphic to $B^4$, which must be standard by \cite{yasha:filling}. Since this 2-handle is attached away from $\Sigma \subset S^1 \times B^3$, the Lagrangian disk $\Sigma \subset S^1 \times B^3$ gives rise to a Lagrangian disk in $B^4$ whose boundary is a Legendrian representative of the slice knot in $S^3$ induced by $B'$. Carving out this Lagrangian disk as indicated in the third diagram of Figure~\ref{fig:lag} yields a Stein domain $W''$ with the same boundary as $W'$ and $W$. Since $W''$ is the exterior of a smooth disk in $B^4$, it embeds smoothly in $\cc^2$. And since $W''$ has the homology type of $S^1 \times B^3$, the embedding $W'' \hookrightarrow \cc^2$ preserves the homotopy class of almost-complex structure on $W''$ and thus $W$ is isotopic to a Stein subdomain of $\cc^2$ by \cite{gompf:image}.

It is straightforward to modify this construction to produce an infinite family of such examples. For example, one may modify the curve $B$ in Figure~\ref{fig:Z1stein} by taking a (local)  connected sum with any nontrivial Legendrian knot that bounds a regular Lagrangian slice disk in $(B^4,\omega_\st)$ \cite{conway-etnyre-tosun,cornwell-ng-sivek:obstructions}. We leave it to the reader to check that the arguments above extend to this family of 4-manifolds.
\end{proof}

\section{Higher genus traces and the  embedding lemma}\label{sec:definitons}\label{subsec:embedding} 

In this section, we define higher genus traces, state and prove the general trace embedding lemma, and collect the necessary background about Stein structures on 4-manifolds. The key definition is the following generalization of a 4-dimensional 2-handle attachment:

\begin{defn}\label{def:hgh}
For any integer $g \geq 0$, a \emph{genus $g$ handle} is a copy of $F \times D^2$, where $F$ is a compact genus $g$ surface with one boundary component, attached to the (outward-normally) oriented boundary of an oriented 4-manifold $X$  by an embedding $\varphi: \partial F \times D^2 \to \partial X$. \end{defn}

As with traditional handle attachments,  there is a canonical way to smooth the corners of $X \cup_\varphi F \times D^2$, and the diffeomorphism type of the resulting smooth 4-manifold is determined by two pieces of data: 
\begin{enumerate}
\item  the knot $K \subset \partial X$ along which $\partial F \times 0 \subset \partial F \times D^2$ is attached, and 
\item a framing of a trivial tubular neighborhood $\nu(K) \subset \partial X$ used to identify $\nu(K)$ with $\partial F \times D^2$ under the handle attachment.
\end{enumerate}

Given a knot $K \subset S^3$, we define its \emph{$n$-framed, genus $g$ trace} $X^g_n(K)$ to be the oriented 4-manifold obtained by attaching an $n$-framed, genus $g$ handle to the oriented $B^4$ along $K$.

\begin{rem}
For a diagrammatic description of higher genus handle attachments, see \S\ref{subsec:diagrams}.
\end{rem}

\subsection{The trace embedding lemma}

The trace embedding lemma proven below is a straightforward generalization of the classical result for knot traces; its proof is well-known to experts. Given a fixed $n$-framed, genus $g$ trace $X^g_n(K)$, we let $\alpha$ denote a generator of second homology. For the following statement, let $W$ be any smooth 4-manifold and let $\beta \in H_2(W)$ be any chosen second homology class.

\begin{lem}[The trace embedding lemma]\label{lem:traceglory}
There exists a smooth embedding $f:X^g_n(K)\to W$ with $f_{*}(\alpha)=\beta \in H_2(W)$ if and only if the mirror $-K$ bounds a smooth genus $g$ surface $\Sigma$ in $W\smallsetminus \mathring{B^4}$ with $[\Sigma]=\beta\in H_2(W \smallsetminus \mathring{B^4}, S^3)\cong H_2(W)$ such that $\beta \cdot \beta=n$.
\end{lem}

\begin{rem}
We assume that the boundary of an oriented 4-manifold is equipped with the orientation induced by the outward normal direction. The consideration of orientations can be important; for example, when $S^3$ is considered as the boundary of $\cp$, the right-handed trefoil bounds a nullhomologous disk and the left-handed trefoil does not. 
\end{rem}

\begin{proof}[Proof of Lemma~\ref{lem:traceglory}]
We begin with the ``if" direction: Since $\Sigma$ is smooth and has nonempty boundary, it has a trivial tubular neighborhood $\nu(\Sigma)\cong \Sigma\times D^2$ in $W\smallsetminus \mathring{B^4}$. Now consider $W$, obtained from $W\smallsetminus \mathring{B^4}$ by gluing on a standard 4-ball. Observe that the union of this 4-ball and $\nu(\Sigma)$ forms a genus $g$ trace $X$ for some knot in $S^3$. The genus $g$ handle is clearly attached along $-K$ when $S^3$ is viewed as the oriented boundary of $W\smallsetminus \mathring{B^4}$. However, traces are defined in terms of the attaching map to $S^3$ with its orientation as the boundary of the 4-ball, so $X$ is in fact the trace of the knot $K \subset \partial B^4$.

It remains to verify the framing on the trace $X$ and the claim about the induced map on second homology. To that end, let $\Sigma'$ be any closed surface obtained by capping off $\Sigma$ with a Seifert surface for $-K$ in the boundary of $W \smallsetminus \mathring{B^4}$. The surface $\Sigma'$ represents $\beta$ under the isomorphism $H_2(W) \cong H_2(W \smallsetminus \mathring{B^4},S^3)$. It also represents a generator $\alpha \in H_2(X)$, and thus the inclusion of $X$ into $W$ maps $\alpha$ to $\beta$, as claimed. Finally, since $[\Sigma']\cdot [\Sigma']= \beta \cdot \beta = n$, we have $\alpha \cdot \alpha = n$ and thus $X=X^g_n(K)$.

For the ``only if" direction: Let us suppose that $X_n^g(K)$ embeds smoothly in $W$, and let $\Sigma' \subset X_n^g(K) \subset W$ be a piecewise-linear representative of $\alpha \in H_2(X^g_n(K))$ formed from the union of the cone on the knot $K$ and the core surface of the genus $g$ handle. A sufficiently small neighborhood $V \cong B^4$ of the cone point meets $\Sigma'$ in a singular disk $\Delta$ with $\partial \Delta=K$ in $\partial V \cong S^3$. Then $\Sigma=\Sigma' \smallsetminus \mathring{\Delta}$ is a smooth, genus $g$ surface in $W \smallsetminus \mathring{V} \cong W \smallsetminus \mathring{B^4}$. Since $\partial \Sigma$ is $K$ when viewed in the boundary of $V \cong B^4$, we see that $\partial \Sigma$ is $-K$ when viewed in the boundary of $W \smallsetminus \mathring{V} \cong W \smallsetminus \mathring{B^4}$. Using the isomorphism $H_2(W \smallsetminus \mathring{B^4},S^3)\cong H_2(W)$, $\Sigma \subset W \smallsetminus \mathring{B^4}$ represents  the image $\beta \in H_2(W)$ of the generator $\alpha \in H_2(X^g_n(K))$ under the map induced by the inclusion of $X^g_n(K)$ into $W$.
\end{proof}

 As mentioned in the introduction, the trace embedding lemma can provide a powerful tool for determining whether or not a given knot $K'$ is slice (e.g.~\cite{picc:conway}). In particular, if two knots $K$ and $K'$ in $S^3$ have diffeomorphic zero-traces, then $K$ is slice if and only if $K'$ is slice. Since the trace embedding lemma holds for higher genus traces, it has been asked if the strategy from \cite{picc:conway} can be used to obstruct the existence of higher genus slice surfaces. Unfortunately, the naive extension of this strategy is doomed, as one cannot even find a pair of inequivalent knots in $S^3$ whose higher genus traces have the same boundary:

\begin{prop}\label{prop:fail}
If $\partial X^g_n(K) \cong \partial X^g_n(K')$ for some integer $g >0$, then there is a homeomorphism of pairs taking $(S^3, K)$ to $(S^3, K')$.
\end{prop}

\begin{proof}
The key idea is that $\partial X^g_n(K)$ splits into two pieces: a knot complement in $S^3$ and the 3-manifold $F \times S^1$, where $F$ is a compact surface with genus $g>0$ and one boundary component.  Since a knot is determined by its complement~\cite{gordon-luecke}, it suffices to show that such a decomposition is unique.

The JSJ decomposition theorem states that $\partial X^g_n(K)$ contains a unique maximal  Seifert fibered submanifold $M$ (the \emph{characteristic submanifold}) whose complement is atoroidal \cite{jaco-shalen,johannson}. Let $N \subset \partial X^g_n(K)$ denote the copy of $F \times S^1$ that replaced the solid torus neighborhood of $K \subset S^3$. Note that $N$ clearly lies in $M$. We will show that $N$ is the unique embedding of $F \times S^1$ in $\partial X^g_n(K)$.  First note that the torus $\partial N$ is incompressible if and only if $K$ is a nontrivial knot. This allows us to eliminate the unknot $U$ from consideration, as $\partial X^g_n(U)$ contains strictly fewer  isotopy classes of incompressible tori than $\partial X^g_n(K)$ for $K \neq U$. 

Thus let us suppose that $K$ is not the unknot, in which case $\partial N$ is incompressible.  Then  $\partial N$ is isotopic to a surface that is either vertical (i.e.~a union of fibers) or horizontal (i.e.~transverse to all fibers); see \cite[Proposition~1.11]{hatcher:3-manifolds}. We can rule out the latter: If $\partial N$ is horizontal, then it meets every fiber in its connected component of $M$ and cuts each fiber into a collection of intervals. In particular, $N$ itself is then an interval bundle whose base is a  surface that is double-covered by the torus $\partial N$. This implies that $N$ must be a twisted interval bundle over a nonorientable closed surface. But any such space has torsion in its first homology, which $N \cong F \times S^1$ does not. 

We may therefore assume that $\partial N$ is a vertical surface, so the Seifert fibration on $M$  restricts to one on $N$. The torus $\partial N$ projects to a simple closed curve that separates the base $B$ of $M$ into subsurfaces $B_1$ and $B_2$ that are bases for the restricted Seifert fibrations on $N$ and $M \smallsetminus \mathring{N}$. Since the unique Seifert fiber structure on $N \cong F \times S^1$ is the product structure, we have $B_1 \cong F$. And since $\partial X^g_n(K) \smallsetminus N$ is the knot complement $S^3 \smallsetminus K$, we see that $M \smallsetminus \mathring{N}$ is a Seifert fibered submanifold of $S^3$, hence $B_2$ is planar; see \cite{budney}.   Thus $B$ is obtained from gluing a planar surface $B_2$ to $B_1 \cong F$ along some component of $\partial B_2$.

Now suppose that $N'$ is another  copy of $F \times S^1$ embedded in $\partial X^g_n(K)$. As above, we conclude that $N'$ lies in $M$ and that its boundary $\partial N'$ is a vertical torus. Moreover, the image of $\partial N'$ under the Seifert fibration $M \to B$ is a separating simple closed curve that bounds a subsurface $B_1' \cong F$ of $B$ (over which $N'$ lies).  However, up to automorphism of $B$, there is a unique embedding of $F$ into $B$, namely $B_1$. Since neither $N$ nor $N'$ contain any exceptional fibers of $M$, we conclude that $M\smallsetminus N\cong M\smallsetminus N'$. It follows that $\partial X^g_n(K)$ has a unique decomposition as the union of $F \times S^1$ and a knot complement in $S^3$. Since a knot is determined by its complement~\cite{gordon-luecke}, it follows that the JSJ decomposition of $\partial X^g_n(K)$ determines the knot $K$.
\end{proof}

\subsection{Handle diagrams and Stein structures} \label{subsec:diagrams}

Let $X$ be a 4-manifold described by a handle diagram. Given a knot $K \subset \partial X$ drawn inside the diagram, we can produce a diagram for the 4-manifold $X'$ obtained by attaching an $n$-framed genus $g$ handle to $X$ along $K$ as follows: First, attach $2g$ 1-handles to $X$ in a small neighborhood of $K \subset \partial X$ as indicated in Figure~\ref{fig:g-handle}. Then attach an $n$-framed 2-handle along the knot in $\partial X \#_{2g}(S^1 \times S^2)$ induced by $K \subset \partial X$. (This construction is easily adapted from the handle decomposition of the disk bundle over a genus $g$ surface with Euler number $n$; see~\cite[Example 4.6.5]{GompfStipsicz4}.)     This can be adapted to the Stein setting as follows:

\begin{lem}\label{lem:stein}
If $X'$ is obtained by attaching a genus $g$ handle to a Stein domain $X$ along a Legendrian knot $K \subset \partial X$ with framing $2g-1$ relative to the contact framing of $K$, then $X'$ admits a Stein structure.
\end{lem}

\begin{rem}\label{rem:framings}
We follow the standard conventions for handlebody diagrams. In particular, if $w(L_i)$ denotes the writhe of a Legendrian link component $L_i$ in a handlebody diagram, then the $n$-framing curve wraps around $L_i$ a total of $(n-w(L_i))$ times relative to the blackboard framing. Moreover, our Stein handlebody diagrams are drawn in the ``standard form'' from~\cite{gompf:stein}. This implies that Thurston-Bennequin number of $L_i$ is given by $tb(L_i)=w(L_i)-\lambda(L_i)$, where $\lambda(L_i)$ denotes the number of left cusps in the Legendrian diagram of $L_i$. If $L_i$ is oriented, we may also use the diagram to define a rotation number $r(L_i)=\tfrac{1}{2}(t_- - t_+)$, where $t_+$ (resp.~$t_-$) denotes the number of upward-oriented (resp.~downward-oriented) cusps in the diagram of $L_i$.
\end{rem}

\begin{figure}
\center
\def\svgwidth{.65\linewidth}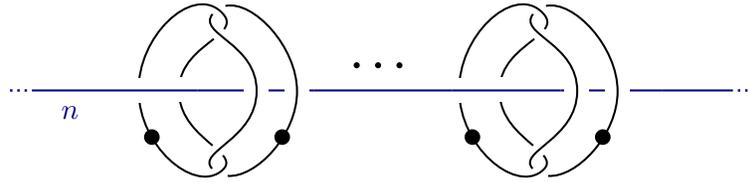 
\caption{An $n$-framed genus $g$ handle attachment using $2g$ 1-handles and one 2-handle.}
\label{fig:g-handle}
\end{figure}

\begin{figure}[b]
\center
\def\svgwidth{.28\linewidth}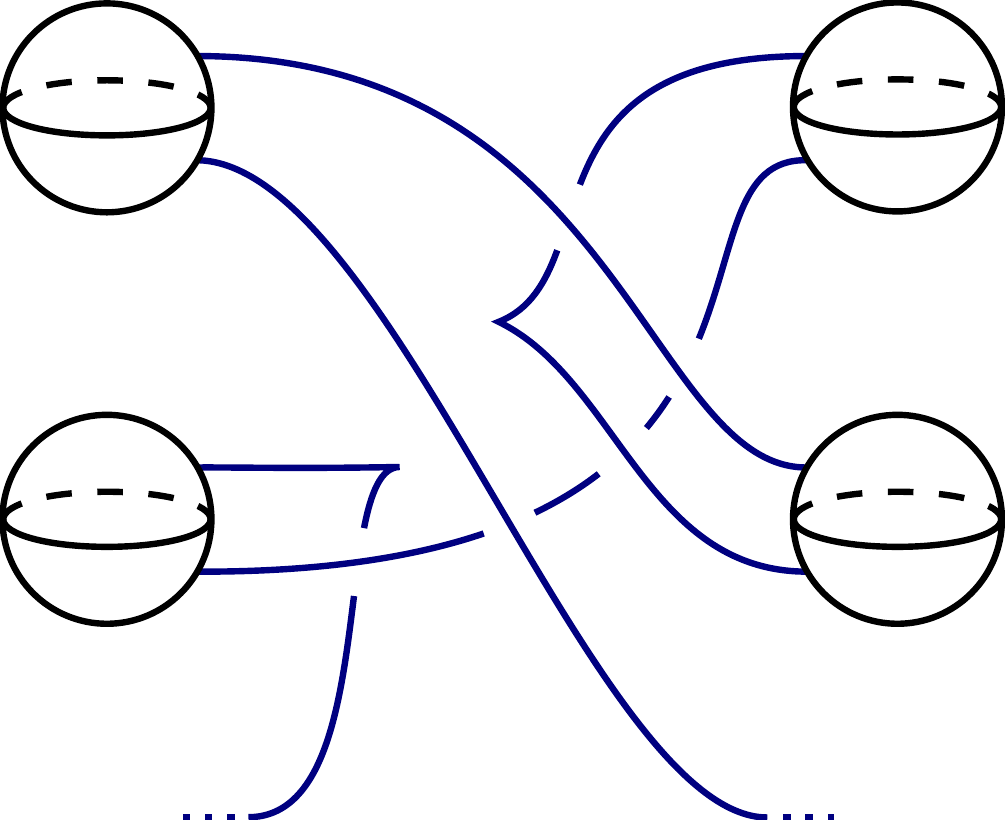 
\captionsetup{width=.9\linewidth}
\caption{In Lemma \ref{lem:stein}, we replace an arc of a Legendrian knot with the given configuration of two 1-handles and an arc traversing the 1-handles algebraically zero times. This increases the writhe by three and adds one left cusp, hence increases the contact framing by 2.}
\label{fig:stein}
\end{figure}
\begin{proof}
Fix a Stein handle diagram for $X$ in which all 2-handles are attached along Legendrian knots in $\natural(S^1 \times B^3)$ in standard form, and let us further assume that $K$ is drawn in standard form in the diagram. Begin by choosing  any smooth arc $\gamma$ of $K$ in the diagram, i.e.~an arc with no cusps or crossings. Attach $g$ pairs of Stein 1-handles in a neighborhood of this arc and, near each such pair of 1-handles, modify $K$ as in Figure~\ref{fig:stein} by replacing a subarc of $\gamma$ with the indicated arc passing through the pair of 1-handles.

Each of these $g$ arcs has writhe equal to $+3$ and contains one left cusp. Though this diagram will not be in standard form, this can be remedied by dragging the 1-handles by their feet out to the sides of the diagram, threading them through the other 2-handle attaching curves. Note that this process does not change the writhe or number of left/right cusps of the new Legendrian knot $\tilde{K}$ in the diagram. Thus we calculate (as in Remark~\ref{rem:framings}) that the contact framing of $\tilde{K}$ differs from that of $K$ by $+2g$. This implies that attaching a genus $g$ handle to $X$ along $K$ with framing $2g-1$ relative to the contact framing of $K$ corresponds to attaching a 2-handle to $X \natural_{2g} (S^1 \times B^3)$ along $\tilde{K}$ with framing $-1$ relative to the contact framing of $\tilde{K}$. It follows that $X'$ admits a Stein structure.
\end{proof}

\begin{rem}\label{rem:rot}
Regardless of the orientation of the Legendrian knot $K \subset \partial X$, an inspection of Figure~\ref{fig:stein} shows that the modified knot $\tilde{K} \subset \partial X \#_{2g} (S^1 \times S^2)$ has $r(\tilde{K})=r(K)$. 
\end{rem}

\begin{rem}
We will be interested in building Stein 4-manifolds by attaching 2-handles along Legendrian knots in the boundary of \emph{the Mazur cork} $M$, which is the contractible Stein 4-manifold defined in Figure \ref{fig:Mazur}.
\end{rem}

\begin{figure}[h]
\center
\def\svgwidth{.45\linewidth}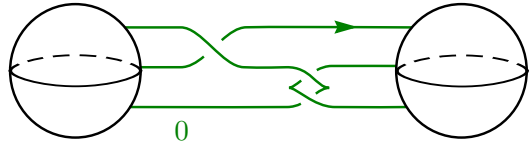 
\caption{The Mazur cork $M$.}
\label{fig:Mazur}
\end{figure}

\section{Spineless 4-manifolds for all genera and intersection forms}
\label{sec:fake-traces}

We now prove a refined statement of  Theorems~\ref{thm:main} and \ref{thm:spineless}.

\begin{thm}\label{thm:detail}
For any integers $g,n \geq 0$, there exist infinitely many 4-manifolds $W$ that are each homotopy equivalent to a closed genus $g$ surface and have intersection form $\langle n \rangle $ such that
\begin{enumerate}[label=\normalfont(\alph*)]

\item $W$ admits a Stein structure,

\item $W$ embeds smoothly in $\#_n \cp$,

\item there is no embedding of a genus $g$ knot trace into $W$  inducing an isomorphism on second homology,  and

\item $W$ is homeomorphic to a smooth 4-manifold $W'$ that admits a smooth spine. When $g=0$, $W'$ can be taken to be the $n$-trace $X_n(K)$ of a slice knot $K$ in $S^3$.
\end{enumerate}
\end{thm}

\textbf{Construction of $\boldsymbol{W}$ and proof of (a)}: 
 We will construct $W$ by attaching an $n$-framed, genus $g$ handle along a certain Legendrian knot in the contact boundary of the Stein Mazur cork $M$. Since $M$ is contractible, the resulting 4-manifold $W$ is homotopy equivalent to a closed genus $g$ surface and has intersection form $\langle n \rangle$. We will arrange for our Legendrian knot to have a prescribed Thurston-Bennequin number so that, by Lemma~\ref{lem:stein}, 
 $W$ admits a Stein structure. Moreover, by  carefully arranging the rotation number 
 of the Legendrian attaching knot, we will control the evaluation of the first Chern class $c_1(W)$ on a preferred generator of $H_2(W) \cong \zz$.

\begin{figure}
\center
\def\svgwidth{.9\linewidth}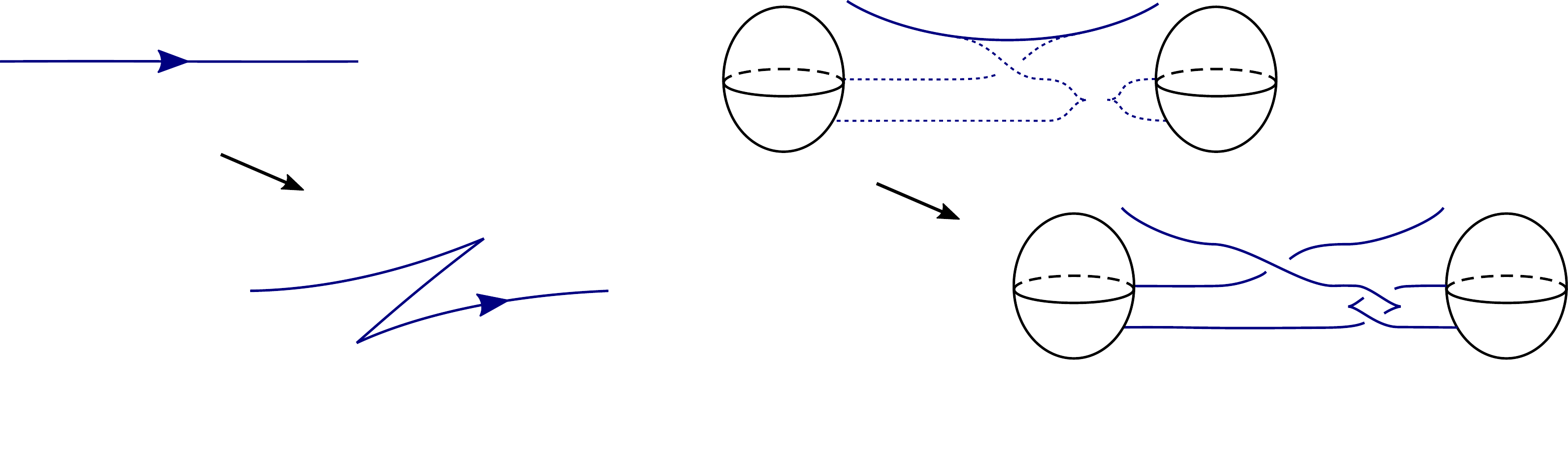 
\caption{(a) A positive Legendrian stabilization. (b) Adding a clasp across a 1-handle.}
\label{fig:claspstab}
\end{figure}

To describe these examples explicitly, we begin by again considering the (blue) oriented knot $B$ in the boundary of $M$ depicted in  Figure~\ref{fig:Z1stein}. Observe that the diagram of $B$ is in standard form with respect to the 1-handle and 2-handle of $M$.  Now choose any positive integer $i$ and perform  $4g+n+2i$ positive stabilizations of $B$ as in Figure~\ref{fig:claspstab}(a). After these stabilizations, the rotation number of $B$ (as defined in Remark~\ref{rem:framings}) is $r(B)=1+4g+n+2i$ and the Thurston-Bennequin number  is $tb(B)=1-4g-n-2i$. Next we modify  the smooth isotopy type of $B$ to achieve the desired Thurston-Bennequin number.   We define \emph{adding a clasp over the 1-handle} to be the modification of the Legendrian arc depicted in Figure~\ref{fig:claspstab}(b). We now modify $B$ by iteratively adding $n+g+i$ clasps to its bottommost strand. This modification increases $tb$ by $2(n+g+i)$ and preserves $r$, thus the new blue curve has $tb(B)=1-2g+n$ and $r(B)=1+4g+n+2i$. 

\begin{figure}
\center
\def\svgwidth{.7\linewidth}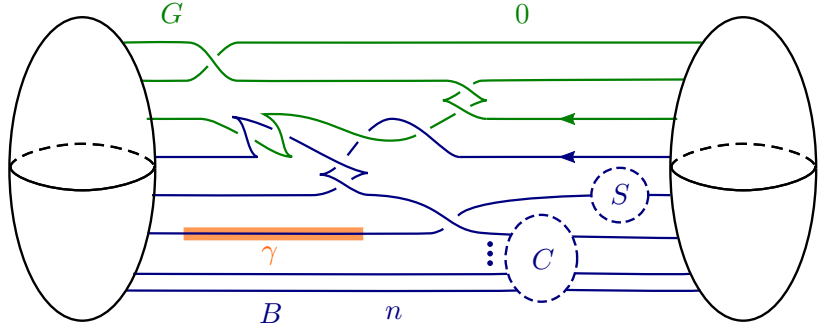 
\caption{A schematic of the 4-manifold $W$. For any nonnegative integer $i$, $W$ has $2g$ 1-handles in the neighborhood $\gamma$ as in Figure \ref{fig:g-handle}, $4g+n+2i$ positive stabilizations in the region $S$ as in Figure \ref{fig:claspstab}(a), and $n+g+i$ clasps in $C$ as in Figure~\ref{fig:claspstab}(b).}
\label{fig:Wschem}
\end{figure} 

Now let $W$ be obtained from $M$ by attaching an $n$-framed, genus $g$ handle along $B$.  Since $n-tb(B)=2g-1$, Lemma~\ref{lem:stein} implies that $W$ admits a Stein structure. See Figure \ref{fig:Wschem} for a schematic picture of $W$. 

Observe that, as in the proof of Lemma~\ref{lem:stein}, we have a natural Stein handle description of $W$  with $(2g+1)$ 1-handles and two 2-handles attached along knots $\tilde G,\tilde B \subset \natural_{2g+1}S^1 \times B^3$, where $\tilde G$ and $\tilde B$ are the knots in the boundary of the 1-handlebody  induced by the knots $G$ and $B$. To understand the  Chern class $c_1(W)$ associated to the Stein structure on $W$, we consider this Stein handle diagram of $W$. Note that $H_2(W)$ is generated by a class $\alpha$ corresponding to the difference of the 2-handles attached along the oriented curves $\tilde B$ and $\tilde G$. Following \cite[Proposition~2.3]{gompf:stein}, the evaluation of $c_1(W)$ on $\alpha$ is given by the difference in the rotation numbers of $\tilde B$ and $\tilde G$. As noted in Remark~\ref{rem:rot}, these coincide with the rotation numbers of the underlying oriented knots $B$ and $G$ in $S^1 \times B^3$. Therefore, we have
\begin{equation}\label{eq:chern}
\langle c_1(W) ,\alpha \rangle = r(B)-r(G) \geq (4g+n+1+2i)-1=4g+n+2i. 
\end{equation}
In particular, by increasing $i$, we obtain a family of Stein domains $W_i$ such that $\langle c_1(W_i),\alpha\rangle$ grows arbitrarily large.
\hfill $\square$

\textbf{Proof of (b)}: 
 First observe that $W$ embeds smoothly in $X_{n}^g(U)$; indeed one recovers $X_{n}^g(U)$ by attaching a 2-handle to $W$ to cancel the 1-handle in the Mazur cork $M\subset W$  and then a 3-handle to cancel the green 2-handle $G$. Note that this embedding $W \hookrightarrow X^g_n(U)$ induces an isomorphism on second homology. Next observe that $U$ bounds a genus $g$ surface $\Sigma$ in $B^4\#_n\cpbar$ such that $\Sigma$ represents the $(1,\ldots,1)$ class in $H_2(B^4 \#_n \cpbar, S^3) \cong \zz^n$. By Lemma~\ref{lem:traceglory}, there exists a smooth embedding $X_{n}^g(U) \hookrightarrow \#_n \cp$ carrying the generator of $H_2(X_{n}^g(U))$ to the $(1,\ldots,1)$ class in $H_2(\#_n \cp)$.  \hfill $\square$

\smallskip

\textbf{Proof of (c)}: 
 For the sake of contradiction, suppose there exists a knot $J\subset S^3$ and a smooth embedding  $X_n^g(J) \hookrightarrow W$ inducing an isomorphism on second homology.  
As argued in the proof of part (b), there exists a smooth embedding $W \hookrightarrow \#_n \cp$ carrying the generator of $H_2(W)$ to the $(1,\ldots,1)$ class in $H_2(\#_n \cp)$.  Composing these embeddings and applying Lemma~\ref{lem:traceglory}, we see that $J$ bounds a smooth genus $g$ surface in $B^4\#_n \cpbar$ representing the class $(1,\ldots, 1)$ in $H_2(B^4 \#_n \cpbar,S^3)$.  Thus the blowup $X^g_n(J) \#_n \cpbar$ contains  a smooth surface $\Sigma$ of genus $2g$ formed from the core of the genus $g$ handle and the genus $g$ surface $J$ bounds in $B^4\#_n \cpbar$. Moreover, $X^g_n(J) \#_n \cpbar$ embeds in $W \#_n \cpbar$, and the  of the surface $\Sigma$ 
 represents the class $\alpha +(1,\ldots,1)$ in $H_2(W \#_n \cpbar)$, where $\alpha$ denotes the preferred generator of $H_2(W)$. Note that, since $\alpha \cdot \alpha =n$ and $(1,\ldots,1)\cdot (1,\ldots,1)=-n$, we have $\Sigma\cdot \Sigma=0$. 

We claim that  $W\#_n \cpbar$ cannot contain such a surface. First, since it is Stein, $W$ admits a holomorphic embedding $\iota$ into a closed 
K\"ahler surface $Z$ of general type with $b_2^+(Z) \geq 2$ \cite{lisca-matic}. In particular, we have $\iota^*c_1(Z)=c_1(W)$, where  $c_1(W)$ and $c_1(Z)$ are the distinguished Chern classes  associated to the complex structures on $W$ and $Z$.   The blowup $Z\#_n \cpbar$ also has distinguished Chern class $c_1(Z)+c$, where $c$ is the Poincar\'e dual to the (1,\ldots,1) class in $H_2(\#_n \cpbar)$.  The Seiberg-Witten invariant of $Z$ associated to $c_1(Z)$ is nonvanishing \cite{witten:monopoles}, and the blow-up formula for Seiberg-Witten invariants implies the same for $c_1(Z)+c \in H_2(Z \#_n \cpbar)$ \cite{fintushel1stern:immersed}. Therefore we have the following adjunction inequality for homologically essential embedded surfaces $F$ of non-negative self-intersection \cite{km:thom, mst:product, fintushel1stern:immersed}: 
\begin{equation}\label{eq:adj}
2g(F)-2 \geq |\langle c_1(Z)+c,[F]\rangle|+F\cdot F.
\end{equation}

The embedding of the 4-manifolds $W \#_n \cpbar\hookrightarrow Z \#_n \cpbar$ gives rise to an embedding 
$\Sigma\hookrightarrow Z \#_n \cpbar$ with $[\Sigma]=\iota_*(\alpha) +(1,\ldots,1)$ and, as remarked above, $\iota^*c_1(Z)=c_1(W)$. So \eqref{eq:adj} gives
\begin{equation*}
2g(\Sigma)-2 \geq|\langle c_1(Z)+c,[\Sigma]\rangle|+\Sigma\cdot\Sigma  = |\langle c_1(W),\alpha \rangle - n| +0 \geq 4g = 2g(\Sigma),
\end{equation*}
where the rightmost inequality appeals to \eqref{eq:chern}. This yields the desired contradiction. \hfill $\square$

\begin{rem}
As remarked after \eqref{eq:chern}, we can actually define an entire family of Stein domains $W_i$ such that $\langle c_1(W_i),\alpha \rangle$ grows linearly with $i$ (without changing $g$ or $n$). For sufficiently large $i$, the adjunction inequality then leads to a lower bound $g_i$ on  the minimal genus of a smoothly embedded surface representing $\alpha \in H_2(W_i)$. Of course, $\alpha$ can always be represented by \emph{some} smoothly embedded surface $\Sigma_i$ in $W_i$. Hence, for fixed $n$ and $g$,  there always exists $j>i$ such that $W_j$ has $g_{j}>g(\Sigma_i)$, hence $W_i$ and $W_j$ are distinguished by their genus functions. This ensures that infinitely many of the 4-manifolds $W_i$ are distinct.
\end{rem}
\begin{figure}\center
\def\svgwidth{.6\linewidth}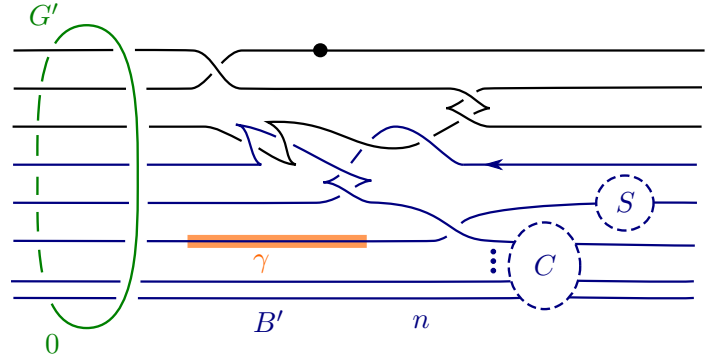 
\caption{A diagram for $W'$ which should be completed by identifying the left edge with the right.}
\label{fig:Zprime}
\end{figure}

\textbf{Proof of (d)}: 
 Consider again the Mazur cork $M$ embedded in $W$, and let $W'$ be obtained from $W$ by a cork twist along $M$ using the involution $\tau:\partial M \to \partial M$ as described in \S\ref{sec:example}. Then $W'$ is homeomorphic to $W$ \cite{akbulut:cork,freedman}. A handle diagram for $W'$ is shown in Figure~\ref{fig:Zprime}, obtained by redrawing Figure~\ref{fig:Wschem} in dotted circle notation  then performing a dot-zero exchange.

We can also view $W'$ as being obtained from $M$ by attaching an $n$-framed genus $g$ handle along the knot in $\partial M$ corresponding to the curve $B'$ in Figure~\ref{fig:Zprime}. Since $B'$ may be isotoped to be unlinked from the dotted circle, it represents a knot in $\partial M$ that bounds a smooth slice disk in $M$. (See the left side of Figure~\ref{fig:Wslide}.)  Thus this slice disk and the core of the genus $g$ handle together form a smooth genus $g$ surface,  easily seen to be a smooth spine of $W'$. 

\begin{figure}\center
\def\svgwidth{\linewidth}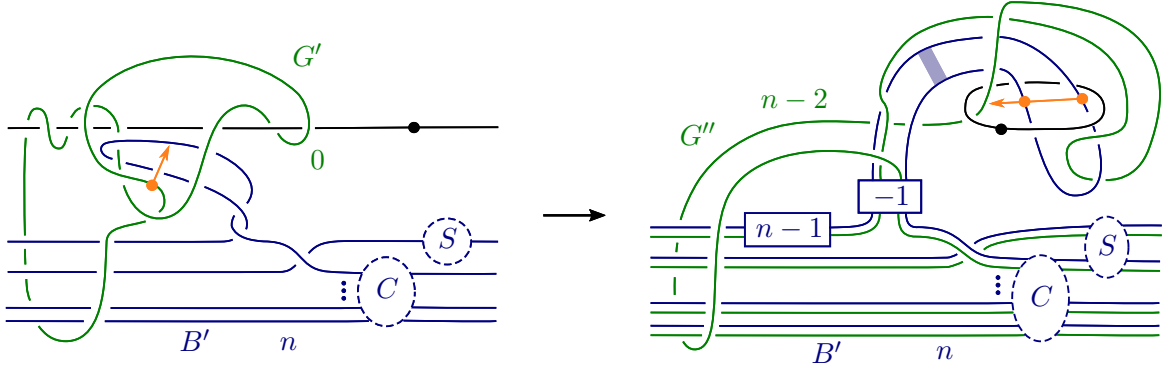 
\caption{On the left, the diagram of $W'$ with $g=0$ has been isotoped to indicate the slide of $G'$ over $B'$ that will turn the 1-handle and $G'$ into a canceling pair. On the right, we indicate the pair of slides which will produce $B''$; after these slides are performed we may cancel $G''$ and the 1-handle.  Labeled boxes denote full twists.}
\label{fig:Wslide}
\end{figure}

For $g=0$, we claim that $W'$ is the $n$-trace $X_n(K)$ of a slice knot $K$ in $S^3$.     
Performing the handleslide of $G'$ over $B'$   indicated in Figure \ref{fig:Wslide}, we obtain a new curve $G''$ that may be isotoped to run over the 1-handle exactly once geometrically. From this diagram we can see that the 1-handle forms a canceling pair with the 2-handle attached along $G''$. In order to perform the cancellation, we first perform the indicated slides of $B'$ over $G''$ to yield a new blue curve that we denote by $K$. After canceling, we obtain a handle structure with exactly one 0-handle and one $n$-framed 2-handle attached along $K$. 

This is a knot trace $X_n(K)$, and we observe that $K$ is ribbon; after the slides and cancellation, performing the band move indicated in the right side of Figure~\ref{fig:Wslide} yields a 2-component unlink. 
\hfill $\square$

\smallskip

This concludes the proof of Theorem~\ref{thm:detail}. For $n \geq 0$,  Theorem~\ref{thm:main} follows immediately; for $n<0$, one simply reverses the orientations of the 4-manifolds constructed in Theorem~\ref{thm:detail}. Theorem~\ref{thm:spineless} follows from Theorem~\ref{thm:detail} just as Corollary~\ref{cor:spineless} followed from Theorem~\ref{thm:main}:

\begin{proof}[Proof of Theorem~\ref{thm:spineless}]
Given a closed, orientable surface $\Sigma$, let $W$ and $W'$ be as in Theorem~\ref{thm:detail} with $g=g(\Sigma)$ and any $n \in\mathbb{N}$. Since $W$ is homeomorphic to $W'$ and the latter contains $\Sigma$ as a smooth spine, $W$ contains a topological locally flat spine given by the image of $\Sigma$ under the homeomorphism.

For the sake of contradiction, suppose that $W$ contains $\Sigma$ as a PL spine. An argument entirely analogous to the one given in the proof of Corollary~\ref{cor:spineless} shows that $\Sigma$ has a neighborhood diffeomorphic to the $n$-framed genus $g$ trace $X^g_n(K')$ of a knot $K'$ in $S^3$.  Since $[\Sigma]$ generates $H_2(W)$, the inclusion $X_n^g(K') \hookrightarrow W$ induces an isomorphism on second homology, contradicting Theorem~\ref{thm:detail}(c).
\end{proof}

\begin{proof}[Proof of Theorem~\ref{thm:approx}]
For $\Sigma$ a closed, orientable surface, consider any smooth embedding $\iota:\Sigma\hookrightarrow Z$, where $Z$ is a smooth 4-manifold. The image $\iota(\Sigma)$ has a tubular neighborhood given by a disk bundle $N$ over $\Sigma$ with Euler number $n=\iota(\Sigma)\cdot \iota(\Sigma)$, the self-intersection number of $\iota(\Sigma)$. The embedding  $\iota:\Sigma \hookrightarrow Z$ corresponds to the inclusion of the zero-section of $N$. We will construct a topological locally flat embedding of $\Sigma$ into $Z$ with image in $N$ that is homotopic to $\iota:\Sigma \hookrightarrow Z$ and show that it cannot be approximated by PL embeddings.

To that end, let $W$ be one of the 4-manifolds from Theorem~\ref{thm:spineless}, where $n=\iota(\Sigma)\cdot \iota(\Sigma)$  and $g=g(\Sigma)$. By Theorem~\ref{thm:spineless}, there exists a topological locally flat embedding $f:\Sigma \hookrightarrow W$ that induces a homotopy equivalence yet is not homotopic in $W$ to any piecewise-linear embedding. As shown in the proof of Theorem~\ref{thm:detail}, $W$ embeds smoothly into the $n$-framed genus $g$ trace $X^g_n(U)$ of the unknot $U$, which is simply the disk bundle $N$ of Euler number $n$ over $\Sigma$. Thus we may let $f_0: \Sigma \hookrightarrow W \hookrightarrow N \subset Z$ denote the topological embedding obtained by composing $f$ with the inclusion of $W$ into $N \subset Z$. Note that the topological embedding $f_0$ induces a homotopy equivalence of $\Sigma$ and $N$, as does the embedding $\iota$ of $\Sigma$  as the zero-section of $N$. Since any two homotopy equivalences $\Sigma \to N$ are homotopic (up to reparametrization by automorphisms of $\Sigma$ \cite[Theorem~8.1]{farb-margalit}), we see that $f_0$ and $\iota$ are homotopic as maps to $N$, hence as maps to $Z$.

For the sake of contradiction, suppose that the topological embedding $f_0: \Sigma \hookrightarrow Z$ can be approximated by PL embeddings. By the definition of the compact-open topology on the space of continuous maps $C^0(\Sigma,Z)$, and since $\Sigma$ is compact and $f_0(\Sigma)$ lies in the open subset $\mathring{W}\subset Z$, there is an open neighborhood of $f_0\in C^0(\Sigma, Z)$ such that all maps in this neighborhood also have image in $\mathring{W} \subset Z$. And since $\Sigma$ is compact, $C^0(\Sigma,Z)$ is locally path-connected; this can be proven directly or by appealing to the fact that $C^0(\Sigma,Z)$ is a Banach manifold \cite{eells}. Therefore any PL embedding $f_1:\Sigma \hookrightarrow Z$ that is sufficiently close to $f_0$ can be connected to $f_0$ by a path in this neighborhood of $f_0 \in  C^0(\Sigma,Z)$. This implies that the original topological embedding $f: \Sigma \hookrightarrow W$ is homotopic (through maps into $W$) to a piecewise-linear embedding, yielding the desired contradiction.
\end{proof}

\section{Pinched surfaces and 2-complexes}\label{sec:2comlex}

In this section we prove Theorem~\ref{thm:2-complex}. The proof calls on basic ideas from piecewise-linear topology, and we refer the reader to \cite{bryant:PL} for background. In particular, recall that a simplicial complex is a triangulation of a piecewise-linear $n$-manifold if and only if the link of every $m$-simplex is homeomorphic to $S^{n-m-1}$ or $B^{n-m-1}$, according to whether $\sigma$ lies in the interior or boundary of the manifold, respectively. (The \emph{link} of a simplex $\sigma$ in a simplical complex $A$ is the subcomplex $\lk(\sigma,A)$ consisting of all simplices $\tau \in A$ disjoint from $\sigma$ such that there is a simplex in $A$ containing both $\sigma$ and $\tau$ as faces.) 

We say that a finite 2-complex $F$ is a \emph{pinched surface} if the link of every vertex in $F$ is homeomorphic to $B^1$ or to a disjoint union of circles $S^1$. The \emph{multiplicity} of a vertex $p\in F$ is the number of connected components in the link of $p \in F$. If $P \subset F$ denotes the set of vertices with multiplicity greater than one, then $F \smallsetminus P$ is a noncompact surface, and we define the \emph{genus} of $F$ to be the genus of $F \smallsetminus P$. (The genus of a disconnected surface is taken to be the sum of the genera of its connected components.)

\begin{prop}\label{prop:pinched}
For any finite 2-complex $C$ and any class $\alpha \in H_2(C)$, there exists a pinched surface $F$ such that for any 4-manifold $X$ and piecewise-linear embedding $\iota : C \hookrightarrow X$, there is a piecewise-linear embedding  $\iota':F \hookrightarrow X$ with $\iota'(F)$ representing  $\iota_*(\alpha)\in H_2(X)$.
\end{prop}

We begin with a lemma:

\begin{lem}\label{lem:nbd}
Let $X$ be a piecewise-linear 4-manifold with boundary and let $(\Delta,\partial \Delta) \hookrightarrow (X,\partial X)$ be a properly embedded, piecewise-linear 2-disk that is locally flat near the boundary. Then there exists a neighborhood $N$ of $\Delta$ that is piecewise-linearly homeomorphic to $B^4$ such that $N \cap \partial X$ is a regular neighborhood of the curve $\partial \Delta$.
\end{lem}

\begin{proof}
Take a path joining all the singular points in $\Delta$, and take a regular neighborhood $B \cong B^4$ of this path so that $B$ meets $\Delta$ in a disk $\Delta' \subset \Delta$. Then $\Delta \smallsetminus \mathring{\Delta}'$ is a locally flat PL annulus that is properly embedded in $X \smallsetminus \mathring{B}$. This annulus has a regular neighborhood $A \subset X \smallsetminus \mathring{B}$ that is  PL homeomorphic to $(\Delta \smallsetminus \mathring{\Delta}') \times D^2 \cong (S^1 \times D^1 ) \times D^2$ and which meets $\partial X$ and $\partial B$ in regular neighborhoods of $\partial \Delta$ and $\partial \Delta'$, respectively. The union $N=A \cup B$ in $X$ is homeomorphic to $B \cong B^4$ and provides the desired neighborhood of $\Delta$.
\end{proof}

\begin{proof}[Proof of Proposition~\ref{prop:pinched}]
For a fixed class $\alpha \in H_2(C;\mathbb{Z})$, we will abstractly define the pinched surface $F$ associated to $(C,\alpha)$. Then we will show that any piecewise-linear embedding of $C$ induces one for $F$.

\emph{Definition of $F$, Step 1.} Choose a representative 2-chain $\sum_{i=1}^m n_i \Delta_i$ for $\alpha$, where each $\Delta_i$ is a 2-simplex in $C$, $\Delta_i = \Delta_j$ only if $i=j$, and $n_i \in \zz$. To begin, let $F_0$ be a disjoint union of 2-simplices $\Delta_{i,j}$ for $i=1,\ldots,m$ and $j=1,\ldots,n_i$, oriented using the sign of $n_i$. Let $p_{i,j}$ be the barycenter of $\Delta_{i,j}$, and let $F_1$ be obtained from $F_0$ by identifying $p_{i,j} \sim p_{i,1}$ for all $1 \leq j \leq n_i$. That is, $F_1$ is obtained by taking $n_i$ disjoint copies of each 2-simplex $\Delta_i$  and pinching them together at a single point in their interiors as in Figure~\ref{fig:pinch}.

\begin{figure}\center
\includegraphics[width=.6\linewidth]{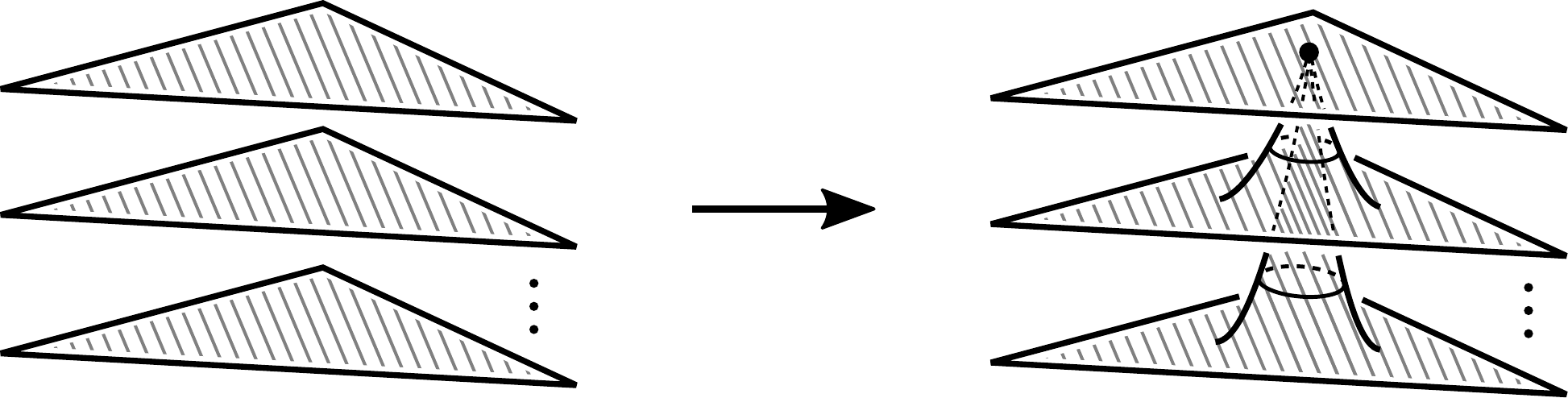}
\caption{Pinch!}\label{fig:pinch}
\end{figure}

\emph{Step 2.} For convenience, we assume that the 1-skeleton $\Gamma$ of $C$ is connected; the general case follows similarly. Choose a collection of ``cutting edges'' $\gamma$ such that $\Gamma$ becomes a tree $\Gamma'$ after deleting the interiors of the chosen edges $\gamma$. (This implies that every PL embedding of $\Gamma'$ into a 4-manifold $X$ will have a neighborhood isotopic to a standard $B^4$, a fact we will use later.) Each cutting edge $\gamma$ has a collection of preimages in $\partial F_1$.  Since $\alpha$ is a cycle, these lifts can be grouped into 	``canceling pairs'' (as in \cite[p.~108-109]{hatcher}). For each such pair of oriented edges in $\partial F_1$ lying over $\gamma$, we attach an oriented band to $F_1$ joining the paired edges in $\partial F_1$; let $F_2$ denote the resulting 2-complex. 

\emph{Step 3.} Finally, let $F$ be obtained as the union of $F_2$ and the cone on $\partial F_2$ (which is a collection of circles).

We now show that given any PL embedding $\iota:C\hookrightarrow X$, we can construct a PL embedding $\iota':F \hookrightarrow X$ so that $[\iota'(F)]=\iota_* (\alpha)\in H_2(X;\Z)$. Since each 2-simplex $\Delta_i \subset C$ is a PL disk in $X$, we may modify the embedding $\iota$ without changing $\iota_*(\alpha)$ so that $\Delta_i$ contains at most one isolated cone point in its interior, which we may view as the point $p_i \in \Delta_i$. We claim that $\Gamma \subset X$ has a (derived) PL regular neighborhood $V \subset X$ with a retraction $r: V \to \Gamma$ such that 
\begin{enumerate}
\item $V$ is PL homeomorphic to $\natural_k S^1 \times B^3$ for $k=b_1(\Gamma)$, and

\item for each edge $\gamma$ in $ \Gamma$, there is an embedded 2-sphere $S_\gamma = \{pt\} \times S^2$ in $\partial V$ lying over $\gamma$ such that each 2-simplex $\Delta$ in $C$ containing $\gamma$ meets $S_\gamma$ in a single point.
\end{enumerate}
The existence of the neighborhood, the retraction, and the homeomorphism in (1) is standard; see \cite{bryant:PL} and the references therein. To prove (2), we recall that $V$ is  obtained by passing to the second barycentric subdivision of $X$ and  taking the union of all simplices that intersect $\Gamma$. If $\sigma$ is a 1-simplex in the interior of the (subdivided) edge $\gamma$ of $\Gamma \subset X$, then $S_\gamma= \lk(\sigma,V)$ is a PL 2-sphere embedded in $\partial V \cong \#_k S^1 \times S^2$ that is mapped to the interior of $\gamma \subset \Gamma$ under the retraction $r: V \to \Gamma$. 
  Moreover, if $\Delta$ is a 2-simplex of $C$ that contains $\gamma$, then $S_\gamma$ meets the (subdivided) 2-simplex $\Delta$ in $\lk(\sigma,\Delta)$, which is a single point in the interior of $\Delta$, as claimed.

Each 2-simplex $\Delta_i$ from the chain representing $\alpha \in H_2(C)$ meets $X \smallsetminus \mathring{V}$ in a smaller 2-disk $\Delta_i'$ obtained from $\Delta_i$ by removing a half-open collar of $\Delta_i$. We may assume $\Delta_i$ is locally flat near $\partial V$, and the intersection $\partial \Delta_i' = \Delta_i \cap \partial V$ is an embedded circle in $\partial V$. Note that if a cutting edge $\gamma$ is a face of $\Delta_i$, then the curve $\partial \Delta_i' \subset \partial V$ meets the 2-sphere $S_\gamma \subset \partial V$ lying over $\gamma$ in exactly one point. By Lemma~\ref{lem:nbd}, the properly embedded disk $\Delta_i'$ in $X \smallsetminus \mathring{V}$ has a neighborhood $N_i \cong B^4$ that meets $\partial V$ in a regular neighborhood of the curve $\partial \Delta_i'$. Moreover, it is clear that  we can choose a disjoint collection of such neighborhoods $N_1,\ldots,N_m$ of $\Delta_1',\ldots,\Delta_m'$. Next, for each $i$, replace the curve $\partial \Delta_i'$ in $\partial V$ with $n_i$ parallel copies of itself. We may now embed the pinched surface $F_1$ from Step 1 into $ \cup_{i=1}^m N_i \subset X \smallsetminus \mathring{V}$ by taking the cone on each collection of $n_i$ parallel copies of $\partial \Delta_i'  \subset \partial N_i$ inside $N_i \cong B^4$.

We will now promote this PL embedding of $F_1$ to a PL embedding of $F_2$ by attaching bands to $\partial F_1$ inside $\partial V$. For each cutting edge $\gamma$ of $\Gamma$ from Step 2, let $S_\gamma$ denote the associated 2-sphere in $\partial V$ satisfying (2). As in Step 2, the edge $\gamma$ has a collection of preimages in $\partial F_1 \subset \partial V$. Moreover, since $S$ intersected $\partial \Delta_i'$ in a single point, it intersects each of these lifts of $\gamma$ (which lie inside the parallel copies of $\partial \Delta_i'$) in $\partial F_1$ in a single point. Thus we may attach the bands from Step 2 to these lifts in an embedded manner in $\partial V$: Any two paired lifts of $\gamma$ in $\partial F_1 \subset \partial V$ intersect $S_\gamma$ in a pair of points. Choose a disjoint collection of embedded paths $\ell$ in $S_\gamma$ such that each path joins two paired intersection points.  For each such path $\ell$, form a band $\ell \times [-\epsilon,\epsilon]$ in a bicollar neighborhood $S_\gamma \times [-\epsilon,\epsilon]\subset \partial V$ joining the given paired lifts of $\gamma$. The union of $F_1$ and these bands gives a PL embedding of $F_2$ in $X \smallsetminus \mathring{V}$.  Note that $\partial F_2$ is disjoint from each such 2-sphere $S_\gamma \subset \partial V$. 

Since $V \cong \natural_k S^1 \times B^3$ is standard as a neighborhood of the graph $\Gamma$,  we may choose properly embedded PL 3-balls $B_\gamma$ in $V \cong \natural_k S^1 \times B^3$ with boundary $\partial B_\gamma = S_\gamma$ such that $B_\gamma$ intersects $\gamma$ in a single point in its interior. (Note that we may need to further subdivide $B_\gamma$ and $V$ to make the embedding piecewise linear.)   By construction, $\partial F_2$ is disjoint from $S_\gamma$ and hence from $B_\gamma$. Moreover, each such 3-ball $B_\gamma$ has an open bicollar neighborhood $B_\gamma \times (-\epsilon,\epsilon)$ meeting $\partial V$ in a neighborhood $S_\gamma \times (-\epsilon,\epsilon)$ disjoint from $\partial F_2$. For each cutting edge $\gamma$, we remove the neighborhood $B_\gamma \times (-\epsilon,\epsilon)$ from $V$. This turns $V$ into a neighborhood $V' \cong B^4$ of the  tree $\Gamma'$ obtained by removing the interior of each cutting edge $\gamma$ from $\Gamma$. Now observe that $\partial F_2$ lies in $\partial V'$. Thus we may embed $F$ into $X$ by taking the union of $F_2 \subset X$ and the cone on $\partial F_2 \subset \partial V'$ inside $V' \cong B^4$. It is straightforward to check that the image of this PL embedding $\iota': F \hookrightarrow X$ represents $\iota_*(\alpha)$ in $H_2(X)$.
\end{proof}

\begin{proof}[Proof of Theorem~\ref{thm:2-complex}]
Given a piecewise-linear 2-complex $C$, we begin by constructing a Stein domain $X$ in $\cc^2$ that is homotopy equivalent to $C$. For convenience, we assume that $C$ is connected. Start with a collection of 0-handles and 1-handles corresponding to the 0- and 1-skeleta of $C$. For each 2-simplex in $C$, let $\gamma$ denote the attaching loop in the 1-skeleton. Choose a knot $K$ in the boundary of the 0- and 1-handles representing the free homotopy class of $\gamma$, and attach a 0-framed 2-handle along $K$. The resulting 4-manifold $X$ has the homotopy type of $C$. Let $\alpha$ denote a fixed generator of $H_2(X;\Z)\cong H_2(C;\Z)\cong \Z$.

We  now further modify $X$ so that it admits the structure of a Stein domain and embeds smoothly in $S^4$. In a handle diagram with dotted circle notation, change crossings of the attaching curves of the 2-handles as needed so that they form an unlink when the dotted circles are erased. Since this corresponds to a homotopy of the attaching maps, the homotopy type of $X$ is unchanged. For convenience, we will also replace $X$ with the boundary connected sum of $X$ and the (contractible) Mazur cork $M$. 

Next put the handle diagram for $X$ in standard position (as in \cite{gompf:stein} and Remark~\ref{rem:framings} above) using any choice of Legendrian representatives of the attaching curves. As in the proof of Theorem \ref{thm:detail}(c), for each attaching curve $K$, we may further modify $X$ (without changing its homotopy type) by performing a sequence of clasp moves across the 1-handle of $M \subset X$ and Legendrian stabilizations of $K$ until $tb(K)=1$. It is straightforward to see that we can do this  in such a way that the attaching curves still form an unlink when the 1-handles are erased. Since  all attaching curves are zero-framed, the modified 4-manifold $X$ admits the structure of a Stein domain. Finally, with $\alpha$ denoting the preferred generator of $H_2(X) \cong \zz$, choose an attaching curve $J$ in the diagram such that any embedded surface representing $\alpha$ has nonzero algebraic intersection number with the cocore of the 2-handle attached along $J$. (Such a 2-handle always exists when $\alpha \in H_2(X)$ is nonzero.) For any integer $k>0$, we again modify $X$ by performing additional clasp moves of $J$ across the 1-handle in $M \subset X$ (followed by pairs of positive stabilizations) until we achieve $\langle c_1(X),\alpha \rangle \geq 2k$.  As in \S\ref{sec:fake-traces}, we may do this in such a way that all 2-handles remain unlinked. Thus the adjunction inequality \eqref{eq:adjunction} for Stein domains ensures that $\alpha$ cannot be represented by a smoothly embedded surface of genus less than $k$. Let $X_k$ denote the resulting Stein domain.

As before, we see that $X_k$ embeds smoothly in $S^4$ by attaching 2-handles corresponding to 0-framed meridians of the dotted circles in the handle diagram of $X_k$; after canceling, the diagram consists of a 0-framed unlink. Attaching 3-handles and a 4-handle yields $S^4$.

Now suppose that $C$ embeds as a piecewise-linear spine of each $X_k$. By Proposition~\ref{prop:pinched}, there exists a fixed pinched surface $F$ such that $\alpha \in H_2(X_k)$ is represented by a piecewise-linearly embedded copy of $F$ in $X_k$. Let $\gamma$ be an embedded path in $F \subset X_k$ which contains all non-manifold points and cone points, chosen to lie in a fixed sheet of $F$ near each non-manifold point. A sufficiently small neighborhood $V \cong B^4$ of this path encloses a collection $\Delta$ of singular pinched disks in $F$, and $F'=F \smallsetminus \mathring{\Delta}=F \cap (S^4 \smallsetminus \mathring{V})$ is a smooth, properly embedded surface in $B^4 \cong S^4 \smallsetminus \mathring{V}$. Moreover, the diffeomorphism type of $F'$ depends only on $F$. We may  replace the singular region $\Delta \subset B^4$ with a copy of (the mirror of) $F'$; this replaces $F$ with a smoothly embedded surface $F'' \subset X_k$ diffeomorphic to the double of $F'$, also representing $\alpha \in H_2(X_k)$. It follows that we must have $k \leq g(F'')$, where the upper bound $g(F'')$ is independent of $k$. That is, for all but finitely many $k>0$, $X_k$ cannot contain $C$ as a piecewise-linear spine.
\end{proof}

\bibliographystyle{gtart}
\bibliography{not-biblio}

\end{document}